\documentclass{amsart}
\pdfoutput=1
	\usepackage{amsmath} %formulars and symbols
	\usepackage{amssymb} %more symbols
	\usepackage{mathtools} %curved arrows
	\usepackage[T1]{fontenc} %if a font-nerd reads my thesis...
	\usepackage[allcolors=black]{hyperref} %clickable links
	\usepackage{bm} %bold symbols
	\usepackage{graphicx} %graphics
	\usepackage{tikz-cd} %commutative diagrams
	\usepackage[noadjust]{cite} %bibliography, noadjust= no silly spaces before citation
	\usepackage{todonotes} %insert todo notes in text
	\usepackage{setspace} %Zeilenabstand Titelseite
	\usepackage{amsthm} %loads the theorem package
	
	\theoremstyle{plain}

		\newtheorem{thm}{Theorem}[section]	
		\newtheorem{lem}[thm]{Lemma}
		\newtheorem{prop}[thm]{Proposition}
		
		\newtheorem{cor}[thm]{Corollary}
	
	\theoremstyle{definition} 
	
		\newtheorem{defn}[thm]{Definition}
		\newtheorem{eg}[thm]{Example}

		\newtheorem{rem}[thm]{Remark}

		\newcommand{\C}{\mathbb C}

		\DeclareMathOperator{\id}{id}

\title[The weak containment problem for \'etale groupoids]{The weak containment problem for \'etale groupoids which are strongly amenable at infinity}

\author{Julian Kranz}
\address{Julian Kranz\\
Westf\"alische Wilhelms-Universit\"at M\"unster,
Mathematisches Institut\\
Einsteinstr.~62, 48149 M\"unster, Germany}
\email{julian.kranz@uni-muenster.de}
\urladdr{https://www.uni-muenster.de/IVV5WS/WebHop/user/j\_kran05/}

\makeatletter
\@namedef{subjclassname@2020}{%
  \textup{2020} Mathematics Subject Classification}
\makeatother

\begin{document}
\maketitle
\begin{abstract}
	We show that an \'etale groupoid which is strongly amenable at infinity is amenable whenever its full and reduced $C^*$-algebras coincide.
\end{abstract}

%\tableofcontents

\section{Introduction}

Let $\mathcal G$ be a locally compact Hausdorff groupoid with a Haar system. We say that $\mathcal G$ \emph{has the weak containment property}, if its full and reduced $C^*$-algebras are isomorphic via the regular representation $\Lambda:C^*\mathcal G\to C^*_r \mathcal G$. It is a classical result  \cite[Proposition 6.1.8]{anantharaman2001amenable} that $\mathcal G$ has the weak containment property whenever $\mathcal G$ is amenable (we will not distinguish between topological and measurewise amenability since they are equivalent for \'etale groupoids \cite[Remark 3.3.9]{anantharaman2001amenable}% or for transformation groups \cite[Corollary 3.29]{buss2020amenability}
). The converse is not true as shown by Willett \cite{willett2015non}. His counterexample is an \'etale groupoid which is not inner exact in the sense of \cite[Definition 3.7]{anantharaman2016some}. However for an exact discrete group $G$ acting on a compact Hausdorff space $X$, Matsumura \cite{matsumura2014characterization} showed that amenability of the transformation groupoid $X\rtimes G$ does follow from the weak containment property. This result was recently generalized to actions of locally compact exact groups on locally compact Hausdorff spaces by Buss, Echterhoff and Willett \cite{buss2020amenability} and to partial actions of exact discrete groups on locally compact spaces by Buss, Ferraro and Sehnem \cite{buss2020nuclearity}. In \cite{anantharaman2016some}, Anantharaman-Delaroche asked whether under some exactness hypothesis, amenability of a groupoid \emph{does} follow from the weak containment property. In this paper, we give a partial answer to her question. Following \cite[Definition 4.1, Proposition 4.8]{anantharaman2016exact}, we call a groupoid $\mathcal G$ \emph{strongly amenable at infinity}, if it acts amenably on its fiberwise Stone-\u{C}ech compactification $\beta_r \mathcal G$. For \'etale groupoids satisfying some mild assumption, this condition is equivalent to a number of other exactness conditions like exactness of the reduced $C^*$-algebra \cite[Theorem 8.6]{anantharaman2016exact}. We emphasize that all groupoids considered in this paper are assumed to be Hausdorff. Our main theorem is the following:

\begin{thm}\label{mainthm}
	Let $\mathcal G$ be an \'etale groupoid which is strongly amenable at infinity. If $C^*\mathcal G=C^*_r \mathcal G$ via the regular representation, then $C^*_r \mathcal G$ is nuclear. 
\end{thm}

In this case, $\mathcal G$ is amenable by \cite[Corollary 6.2.14, Theorem 3.3.7]{anantharaman2001amenable}. \\

The proof of our main theorem follows the same idea as \cite{matsumura2014characterization}. The goal is to factor the inclusion of $C^*_r\mathcal G$ into its double dual through the nuclear $C^*$-algebra $C^*_r(\beta_r\mathcal G\rtimes\mathcal G)$. If we want to imitate the construction of \cite{matsumura2014characterization}, we need to extend the action of $\mathcal G$ on its unit space $X$ to the double dual $C_0(X)^{**}$. This might not be possible since the canonical inclusion $C_0(X)\hookrightarrow C_0(X)^{**}$ is usually degenerate. But since we only consider \'etale groupoids, we can reformulate the problem in terms of partial actions of inverse semigroups. We show that a partial action of an inverse semigroup on a $C^*$-algebra naturally extends to a partial action on the double dual. In particular, we get a partial action on $C_0(X)^{**}$ by the inverse semigroup of open bisections of $\mathcal G$. We then show that the double dual of a partial action is covariantly represented on its Haagerup standard form \cite{haagerup1976standard}. For partial \emph{group} actions, this has already been done in \cite{buss2020nuclearity}. With the Haagerup standard form at hand, we run essentially the same proof as in \cite{matsumura2014characterization} to produce a completely positive contractive map
	\[C^*_r(\beta_r \mathcal G\rtimes \mathcal G)\to (C^*_r(\mathcal G))^{**}\]
which extends the inclusion on $C^*_r(\mathcal G)$.\\

%In the case that $\mathcal G$ is a group, the the strategy in \cite{matsumura2014characterization,buss2020amenability,buss2020nuclearity} involve extending $C_0(X)^{**}\to \mathcal B(H)$ to a covariant representation and then using Arveson's extension theorem to produce an equivariant completely positive map $\ell^\infty(\mathcal G)\otimes C_0(X)\to C_0(X)^{**}$ that extends the identity on $1\otimes C_0(X)$. 
%
%We would like to adapt this strategy to the groupoid case, but we are faced with two problems. First, the inclusion $C_0(X)\to C_0(X)^{**}$ is usually degenerate, so it is not clear in what sense $C_0(X)^{**}$ could carry a $\mathcal G$-action. Second, it is not clear to the author how to construct a canonical $\mathcal G$-action on the Hilbert-space $H$. To circumvent these problems, we simply restrict our attention to \'etale groupoids and make use of the fact that they are crossed products by inverse semigroup actions in a natural way. Indeed, using techniques similar to those in \cite{buss2020nuclearity}, we construct a partial action of the bisection inverse semigroup of $\mathcal G$ on $C_0(X)^{**}$ and extend the representation $C_0(X)^{**}\to \mathcal B(H)$ to a covariant representation. 

The paper is organized as follows:
In Section \ref{sec-etale} we fix some notation concerning groupoid actions on $C^*$-algebras. In Section \ref{sec-inverse} we translate Section \ref{sec-etale} to the context of inverse semigroups. The enveloping von Neumann algebra of a partial action and its Haagerup standard form are introduced in Section \ref{sec-haagerup}. In the last section, we prove our main theorem.

\subsection*{Notation}
The fiber product of two maps $f:X\to Z,g:Y\to Z$ is denoted by $X\times_{f,Z,g}Y:=\{(x,y)\in X\times Y, f(x)=g(y)\}$. If the maps $f$ and $g$ are clear from the context, we omit one or both of them from the notation.

\section{\'Etale groupoids}\label{sec-etale}

	A \emph{groupoid} $\mathcal G$ is a small category in which every morphism is invertible. We denote the set of all morphisms again by $\mathcal G$ and the set of objects by $\mathcal G^{(0)}$, considered as a subspace of $\mathcal G$ via the identity morphisms. $\mathcal G^{(0)}$ is also called the \emph{unit space}. The range and source maps are denoted by $r,s:\mathcal G\to \mathcal G^{(0)}$. For $x\in \mathcal G^{(0)}$, we write $\mathcal G^x:=r^{-1}(x)$ and $\mathcal G_x:=s^{-1}(x)$. A \emph{topological groupoid} is a groupoid $\mathcal G$ together with a topology on $\mathcal G$ such that the range and source maps $\mathcal G\to \mathcal G^{(0)}$, the inverse map $\mathcal G\to \mathcal G$ and the composition map $\mathcal G\times_{s,\mathcal G^{(0)},r}\mathcal G\to \mathcal G$ are continuous. A topological groupoid $\mathcal G$ is called \emph{\'etale}, if it is locally compact Hausdorff and if the range and source maps are local homeomorphisms. In this case, the unit space is clopen in $\mathcal G$ and the fibers $\mathcal G^x\subseteq \mathcal G, x\in \mathcal G^{(0)}$ are discrete. In this article, we only consider \'etale groupoids. We refer to \cite{sims2020groupoids} for an introduction to \'etale groupoids. \\
	
We now introduce actions of \'etale groupoids on $C^*$-algebras. A \emph{$C_0(X)$-algebra} is a $C^*$-algebra $A$ together with a non-degenerate $*$-homomorphism $C_0(X)\to ZM(A)$ into the center of the multiplier algebra of $A$. Equivalently, $A$ is the section algebra of an upper semicontinuous $C^*$-bundle over $X$ (see \cite[Appendix C]{williams2007crossed} for an account on this perspective). The fiber of this bundle at a point $x\in X$ is given by $A_x:=A/(C_0(X\setminus\{x\})A)$. Note that any $C_0(X)$-linear $*$-homomorphism $\phi:A\to B$ of $C_0(X)$-algebras canonically induces $*$-homomorphisms $\pi_x:A_x\to B_x$ on the fibers. We denote by $A\otimes_{C_0(X)}B$ the quotient of the minimal tensor product $A\otimes B$ by the closed two-sided ideal generated by elements of the form $fa\otimes b-a\otimes fb,a\in A,b\in B,f\in C_0(X)$. 

Let $\mathcal G$ be an \'etale groupoid with unit space $X=\mathcal G^{(0)}$. By pulling back along the maps $r,s:\mathcal G\to X$, we equip $C_0(\mathcal G)$ with the structure of a $C_0(X)$-algebra. The algebras 
	\[r^*A:=C_0(\mathcal G)\otimes_{r,C_0(X)}A,\quad s^*A:=C_0(\mathcal G)\otimes_{s,C_0(X)}A\]
are $C_0(\mathcal G)$-algebras where the subscript $r$ or $s$ in the tensor product indicates the $C_0(X)$-structure on $C_0(\mathcal G)$ that we are considering. 

\begin{defn}[\cite{le1999theorie}]
	 A \emph{$\mathcal G$-$C^*$-algebra} $(A,\alpha)$ is a $C_0(X)$-algebra $A$ together with a $C_0(\mathcal G)$-linear $*$-isomorphism $\alpha:s^*A\to r^*A$ such that for all $(g,h)\in \mathcal G\times_{s,X,r}\mathcal G$ we have 
		\[\alpha_{gh}=\alpha_g\circ \alpha_h:A_{s(h)}\to A_{r(g)}.\]
	A $C_0(X)$-linear $*$-homomorphism $\phi:A\to B$ between $\mathcal G$-$C^*$-algebras $(A,\alpha)$ and $(B,\beta)$ is called \emph{equivariant}, if the following diagram commutes:
		\[\begin{tikzcd}
			s^*A\arrow[r,"\id\otimes \phi"]\arrow[d,"\alpha"]	&s^*B\arrow[d,"\beta"]\\
			r^*A\arrow[r,"\id\otimes \phi"]			&r^*B
		\end{tikzcd}\]
\end{defn}

\begin{eg}
	We can reformulate the definition of a commutative $\mathcal G$-$C^*$-algebra $A=C_0(Y)$ in terms of Gelfand duals as follows: The non-degenerate $*$-homomorphism $C_0(X)\to ZM(C_0(Y))= C_b(Y)$ corresponds to a continuous map $p:Y\to X$ which is also called the \emph{anchor map}. The $*$-isomorphism $\alpha:s^*C_0(Y)\to r^*C_0(Y)$ corresponds to a continuous map $\alpha':\mathcal G\times_{s,X,p}Y\to \mathcal G\times_{r,X,p}Y$ which commutes with the projection onto $\mathcal G$ and which satisfies 
		\[\alpha'_{gh}=\alpha'_g\circ \alpha'_h:p^{-1}(s(h))\to p^{-1}(r(g))\]
	for all $(g,h)\in \mathcal G\times_{s,X,r}\mathcal G$. Here $\alpha'_g$ denotes the restriction of $\alpha'$ to the preimage of $g$ under the projection of $\mathcal G\times_{s,X,p}Y$ resp. $\mathcal G\times_{r,X,p}Y$ onto $\mathcal G$. We also say that $(Y,p,\alpha')$ is a \emph{$\mathcal G$-space}.
\end{eg}

\begin{eg}
	The unit space $X$ of $\mathcal G$ itself is a $\mathcal G$-space where the action $\mathcal G\times_{s,X}X\to \mathcal G\times_{r,X}X$ is given by $(g,s(g))\mapsto (g,r(g))$.
\end{eg}

\begin{defn}[\cite{anantharaman2016exact,anantharaman2014fibrewise}]
	The \emph{fiberwise Stone-\u{C}ech compactification of $\mathcal G$} is the Gelfand dual $\beta_r\mathcal G$ of the commutative $C^*$-algebra of all continuous bounded functions $f:\mathcal G\to \C$ such that for every $\varepsilon>0$ there exists a compact subset $C\subseteq X$ satisfying $|f(g)|<\varepsilon$ for all $g\notin r^{-1}(C)$. We define a $\mathcal G$-action on $C_0(\beta_r \mathcal G)$ by taking the canonical inclusion
		\[\iota:C_0(X)\hookrightarrow C_0(\beta_r\mathcal G),\quad \iota(f)(g):= f(r(g)),\quad f\in C_0(X),g\in \mathcal G\]
	for the $C_0(X)$-structure and by defining the action 
		\[\alpha:C_0(\mathcal G)\otimes_{s,C_0(X),r}C_0(\beta_r\mathcal G)\to C_0(\mathcal G)\otimes_{r,C_0(X),r}C_0(\beta_r\mathcal G)\]
	via the formular $\alpha(f\otimes f')(g,h):=f(g)f'(g^{-1}h)$. Here we identify the codomain of $\alpha$ with certain functions on $\mathcal G\times_{r,X,r}\mathcal G$. 
\end{defn}

Note that the inclusion $\iota:C_0(X)\hookrightarrow C_0(\beta_r \mathcal G)$ is $\mathcal G$-equivariant.
Despite of its name, the fiberwise Stone-\u{C}ech compactification $\beta_r \mathcal G$ is in general \emph{not} compact and its fibers might not agree with the Stone-\u{C}ech compactifications of the range fibers of $\mathcal G$. However in the case that $\mathcal G$ is a group (or more generally if $X$ is compact), $\beta_r \mathcal G$ agrees with the usual Stone-\u{C}ech compactification of $\mathcal G$. 

\begin{defn}
	Let $(A,\alpha)$ be a $\mathcal G$-$C^*$-algebra. Denote by $C_c(\mathcal G,A):= C_c(\mathcal G)r^*A$ the space of compactly supported continuous sections of the upper semicontinuous $C^*$-bundle $\mathcal G\times_{r,X}A$. We define a multiplication and involution on $C_c(\mathcal G,A)$ by 
		\[f*f'(g):=\sum_{h\in \mathcal G^{r(g)}} f(h)\alpha_h(f'(h^{-1}g)),\quad f^*(g):=\alpha_{g}(f(g^{-1})^*)\]
	for $f,f'\in C_c(\mathcal G,A)$ and $g\in \mathcal G$. 
\end{defn}

The \emph{full crossed product} $A\rtimes \mathcal G$ is by definition the enveloping $C^*$-algebra of $C_c(\mathcal G,A)$ (c.f. \cite[Proposition 3.2]{quigg1999c}). To define the reduced crossed product, fix $x\in X$ and consider the Hilbert-$A_x$-module $\ell^2(\mathcal G_x,A_x)$. We define a representation
	\[\Lambda_x:C_c(\mathcal G,A)\to \mathcal L(\ell^2(\mathcal G_x,A_x)),\quad \Lambda_x(f)\xi(h):=\sum_{g\in \mathcal G^{r(h)}}\alpha_{h^{-1}} (f(g))\xi(g^{-1}h).\]
	The \emph{reduced crossed product} $A\rtimes_r \mathcal G$ is defined as the completion of $C_c(\mathcal G,A)$ by the norm 
		\[\|f\|_r:=\sup_{x\in X}\|\Lambda_x(f)\|,\quad f\in C_c(\mathcal G,A).\]
	We get a canonical quotient map $\Lambda:A\rtimes \mathcal G\to A\rtimes_r \mathcal G$. In the case $A=C_0(X)$, we simply write $C^*\mathcal G:=C_0(X)\rtimes \mathcal G$ and $C^*_r\mathcal G:=C_0(X)\rtimes_r \mathcal G$. 

%\todo{we really do not need crossed products, we only need semidirect product groupoids.} 

\section{Inverse semigroups}\label{sec-inverse}

\begin{defn}
	An \emph{inverse semigroup} is a semigroup $S$ such that for every $s\in S$, there exists a unique element $s^*\in S$ such that $ss^*s=s$ and $s^*=s^*ss^*$. 	
\end{defn}

A unit in $S$ is an element $1\in S$ such that $s=1s=s1$ for all $s\in S$. In this article, all inverse semigroups are assumed to have a unit. Of course groups are examples of inverse semigroups. Our main example is the following.

\begin{eg}
	Let $\mathcal G$ be a topological groupoid. A subset $U\subseteq \mathcal G$ is called a \emph{bisection}, if the restrictions of the range and source maps to $U$ are homeomorphisms onto their images. If $U$ and $V$ are bisections, their product 
		\[UV:=\{gh:g\in U,h\in V,r(h)=s(g)\}\]
	is again a bisection. Also the inverse $U^*:=U^{-1}$ of a bisection is again a bisection. With these operations, the collection of all open bisections of $\mathcal G$ becomes an inverse semigroup. It has a unit given by the whole unit space $X$.
\end{eg}

Note that a locally compact Hausdorff groupoid is \'etale if and only if its topology has a basis consisting of open bisections.

\begin{defn}
	Let $S$ be an inverse semigroup with unit $1\in S$ and let $X$ be a set. A partial action $\theta=((X_s)_{s\in S},(\theta_s)_{s\in S})$ of $S$ on $X$ consists of a collection $(X_s)_{s\in S}$ of subsets of $X$ together with bijections 
		\[\theta_s:X_{s^*}\to X_s,\quad s\in S\]
	satisfying
		\begin{enumerate}
			\item $X_1=X$ and $\theta_1=\id$.
			%\item $\theta_s^{-1}= \theta_{s^*}$ for every $s\in S$.
			%\item $\theta_{s^*}(X_s\cap X_{t^*})\subseteq X_{(ts)^*}$ for every $s,t\in S$.
			\item For every $s,t\in S$, we have $\theta_{s^*}(X_s\cap X_{t^*})\subseteq X_{(ts)^*}$.
			\item $\theta_{ts}$ extends $\theta_t\theta_s$ on $\theta_{s^*}(X_s\cap X_{t^*})$.
		\end{enumerate}
\end{defn}

\begin{defn}
	Let $S$ be an inverse semigroup with partial actions $\theta$ and $\omega$ on sets $X$ and $Y$. A map $f:X\to Y$ is called \emph{equivariant} if we have $f(X_s)\subseteq Y_s$ and $\omega_s \circ f= f\circ \theta_s$ on $X_{s^*}$ for every $s\in S$. 
\end{defn}

Depending on the additional structure that $X$ carries, we require some extra conditions on $\theta$: If $X$ is a $C^*$-algebra, all the $X_s$ are required to be closed two-sided ideals and the $\theta_s$ are required to be $*$-isomorphisms. If $X$ is a von Neumann algebra, the $X_s$ are required to be ultraweakly closed two-sided ideals and the $\theta_s$ are required to be $*$-isomorphisms. If $X$ is a topological space, the $X_s$ are required to be open subsets and the $\theta_s$ are required to be homeomorphisms. If $X$ is a Hilbert space, the $X_s$ are required to be closed linear subspaces and the $\theta_s$ are required to be isometries. 

Note that partial actions on commutative $C^*$-algebras and locally compact spaces can be identified with one another via Gelfand duality. We will also call a partial action of $S$ on a Hilbert space $H$ a \emph{partial representation} and identify it with a map  $S\to \mathcal B(H)$ acting by partial isometries. 

\begin{eg}\label{groupoids-give-partial-actions}
	Let $\mathcal G$ be an \'etale groupoid with unit space $X$ and let $(A,\alpha)$ be a $\mathcal G$-$C^*$-algebra. Let $S$ be the inverse semigroup of open bisections of $\mathcal G$. There is a partial action $\theta=((A_U)_{U\in S},(\theta_U)_{U\in S})$ of $S$ defined as follows: For an open bisection $U\subseteq \mathcal G$, define $A_U$ to be the ideal $C_0(r(U))A\subseteq A$. Note that we can identify $A_U$ with $C_0(U)r^*A$ and $A_{U^*}$ with $C_0(U)s^*A$. Under this identification, we define $\theta_U$ to be the restriction
		\[\theta_U:=\alpha|_U:C_0(U)s^*A\xrightarrow{\simeq}C_0(U)r^*A.\]		
	of $\alpha$ to $U$. In particular, there is a canonical partial action of $S$ on $A=C_0(X)$. In this case, we have $A_U=C_0(r(U))$ and $A_{U^*}=C_0(s(U))$ and $\theta_U$ is induced by the canonical homeomorphism $s(U)\xrightarrow{s|_U^{-1}}U\xrightarrow{r}r(U)$. 
\end{eg}
The next two definitions are due to \cite{sieben1997c}.
\begin{defn}\label{covariantrep}
	Let $\theta$ be a partial action of an inverse semigroup $S$ on a $C^*$-algebra $A$. A \emph{covariant representation} of $(A,S,\theta)$ on a Hilbert space $H$ is given by a pair $(\pi,v)$ where $\pi:A\to \mathcal B(H)$ is a $*$-homomorphism and $v:S\to \mathcal B(H)$ is a partial representation such that 
		\begin{enumerate}
			\item $\pi(A_s)H=v_{ss^*}H$ for every $s\in S$.
			\item For every $s\in S$ and $a\in A_{s^*}$, we have 
				\[\pi(\theta_s(a))= v_s\pi(a)v_{s^*}.\]
		\end{enumerate}
\end{defn}

\begin{defn}
	Let $\theta$ be a partial action of an inverse semigroup $S$ on a $C^*$-algebra $A$. Let $C_c(S,A)$ be the set of all finite formal linear combinations $\sum_{s\in S}a_su_s$ where $a_s\in A_s$. We define a product and involution on $C_c(S,A)$ by linear extension of the formulas
		\[(au_s)(bu_t):= \theta_s(\theta_{s^*}(a)b)u_{st},\quad (au_s)^*:=\theta_{s^*}(a^*)u_{s^*},\quad s,t\in S,a\in A_s,b\in A_t.\]
	Let $(\pi,v)$ be a covariant representation of $(A,S,\theta)$ on $H$. The \emph{integrated form} of $(\pi,v)$ is the representation
		\[\pi\rtimes v: C_c(S,A)\to \mathcal B(H),\quad \pi\rtimes v(au_s):=\pi(a)v_s.\]
	Denote by $A\rtimes S$ the Hausdorff completion of $C_c(S,A)$ by the seminorm
		\[\|f\|:=\sup \|\pi\rtimes v(f)\|,\quad f\in C_c(S,A)\]
	where the supremum runs over all covariant representations $(\pi,v)$ of $(A,S,\theta)$. 
\end{defn}

\begin{rem}\label{inclusion-into-crossed-product}
	Since we assumed $S$ to have a unit $1\in S$, we can consider $A$ as a subalgebra of $A\rtimes S$ via the inclusion $a\mapsto a u_1$. For $a\in A_s$, we can even identify $a u_1$ with $a u_{ss^*}$ since the integrated forms of all covariant representations agree on these two elements. 
\end{rem}

As in the group case, we have 
\begin{prop}[{\cite[Proposition 4.8]{sieben1997c}}]\label{covariantrepresentationprop}
	Every non-degenerate representation of $A\rtimes S$ is the integrated form of a covariant representation.
\end{prop}

The following theorem allows us to translate Theorem \ref{mainthm} to the inverse semigroup setting.

\begin{thm}[{\cite[Theorem 7.2]{quigg1999c}}]
	Let $\mathcal G$ be an \'etale groupoid and $A$ a $\mathcal G$-$C^*$-algebra. Equip $A$ with the canonical partial action of the inverse semigroup $S$ of open bisections of $\mathcal G$  as in Example \ref{groupoids-give-partial-actions}. Then there is a canonical isomorphism
		\[A\rtimes \mathcal G\cong A\rtimes S.\]
\end{thm}

\section{The Haagerup standard form}\label{sec-haagerup}

As in \cite{matsumura2014characterization,buss2020amenability,buss2020nuclearity}, a key ingredient for Theorem \ref{mainthm} is Haagerup's standard form of von Neumann algebras. Recall that a cone $P\subseteq H$ in a Hilbert space $H$ is called \emph{self-dual}, if it coincides with its dual $P^\circ:=\{\xi\in H : \langle \xi,\eta\rangle \geq 0 \quad \forall \eta\in P\}$.

\begin{thm}[{\cite[Theorem 1.6]{haagerup1976standard}}]
Let $M$ be a von Neumann algebra. Then there is a Hilbert space $H$, an embedding $M\subseteq \mathcal B(H)$, a conjugate linear isometric involution $J:H\to H$ and a self-dual cone $P\subseteq H$ such that the following properties are satisfied:
	\begin{enumerate}
		\item $JMJ=M'$.
		\item $JcJ=c^*$ for all $c\in Z(M)$ .
		\item $J\xi =\xi$ for all $\xi \in P$.
		\item $aJaJ(P)\subseteq P$ for all $a\in M$.
	\end{enumerate}
\end{thm}

The quadruple $(M,H,J,P)$ is called a \emph{standard form}. It is unique in the following sense:

{\begin{thm}[{\cite[Theorem 2.3]{haagerup1976standard}}]\label{uniqueness}
	Let $(M_i,H_i,J_i,P_i)$ be standard forms for $i=1,2$ and let $\phi:M_1\to M_2$ be a $*$-isomorphism. Then there is a unique unitary $U:H_1\to H_2$ such that 
	\begin{enumerate}
		\item $\phi(x)=UxU^*$ for all $x\in M_1$.
		\item $J_2U=UJ_1$.
		\item $P_2=U(P_1)$.
	\end{enumerate}
\end{thm}

The following Lemma is the inverse semigroup analogue of \cite[Proposition 3.4]{buss2020nuclearity}.

\begin{lem}\label{covarianthaagerup}
	Let $\theta$ be a partial action of an inverse semigroup $S$ on a von Neumann algebra $M$. Let $(M,H,J,P)$ be its standard form and denote by $\iota:M\hookrightarrow \mathcal B(H)$ the inclusion. Then there is a canonical partial representation $v:S\to \mathcal B(H)$ such that $(\iota, v)$ is a covariant representation.

	\begin{proof}
		For $s\in S$ denote by $p_s\in M$ the central projection such that $M_s= p_s M$. It follows from Lemma 2.6 of \cite{haagerup1976standard} that $(p_sM, p_s H, p_sJp_s,p_s(P))$ again is a standard form. Theorem \ref{uniqueness} applied to $\theta_s:p_{s^*}M\to p_s M$ provides us with a unique isometry $v_s:p_{s^*}H\to p_sH$ satisfying
		\begin{enumerate}
			\item $\theta_s(a)=v_sav_{s^*}$ 
			\item $p_s J v_s = v_s J p_{s^*}$ 
			\item $p_s(P)=v_s(p_{s^*}(P))$ 
		\end{enumerate}
		for all $s\in S$ and $a\in M_{s^*}$. Another application of Theorem \ref{uniqueness} shows that the map $s\mapsto v_s$ defines a partial representation of $S$ on $H$. It immediately follows from the above properties that $(\iota,v)$ is covariant.
	\end{proof}

\end{lem}

We will later apply the above lemma to the following construction:

\begin{defn}
	Let $\theta$ be a partial action of an inverse semigroup $S$ on a $C^*$-algebra $A$. We extend $\theta$ to a partial action $\theta^{**}$ on the enveloping von Neumann algebra $M:=A^{**}$ as follows. For $s\in S$, let $M_s:=(A_s)^{**}$ be the enveloping von Neumann algebra of $A_s$, considered as an ultraweakly closed ideal in $M$. Define $\theta_s^{**}:A_{s^*}^{**}\to A_s^{**}$ as the normal extension of $\theta_s$. 
\end{defn}

By uniqueness of the various normal extensions, $\theta^{**}$ is indeed a partial action.

\section{Proof of the main theorem}

The following application of Stinespring's theorem can be proved exactly as \cite[Lemma 4.8]{buss2019injectivity}. 
\begin{lem}
	Let $\theta$ be a partial action of an inverse semigroup $S$ on a $C^*$-algebra $A$. Let $(\phi,v)$ be a completely positive  covariant representation of $A$ on a Hilbert space $H$ (i.e. the same conditions as in Definition \ref{covariantrep} hold with "$*$-homomorphism" replaced by "completely positive map"). Then the map 
		\[\tilde \phi:C_c(S,A)\to \mathcal B(H),\quad a u_s\mapsto \phi(a)v_s\]
	extends to a completely positive map $\tilde \phi:A\rtimes S\to \mathcal B(H)$. If $\phi$ is contractive, then so is $\tilde \phi$.
\end{lem}

\begin{cor}\label{completelypositive}
	Let $A$ and $B$ be $C^*$-algebras equipped with partial actions of an inverse semigroup $S$. Let $\phi:A\to B$ be an equivariant completely positive map. Then the map 
		\[\tilde \phi:C_c(S,A)\to C_c(S,B),\quad a u_s\mapsto \phi(b)u_s\]
	extends to a completely positive map $A\rtimes S\to B\rtimes S$. If $\phi$ is contractive, then so is $\tilde \phi$.
	\begin{proof}
		Take a non-degenerate and faithful representation $B\rtimes S\subseteq \mathcal B(H)$. By Proposition \ref{covariantrepresentationprop}, any such representation is the integrated form of a covariant representation $(\pi,v)$. Now apply the above lemma to the pair $(\pi\circ \phi,v)$. 
	\end{proof}
\end{cor}

%\todo{make multiplicative domain lemma}

\begin{lem}[{\cite[Proposition 1.5.7]{brown2008textrm}}]\label{multiplicativedomain}
	Let $\phi:A\to B$ be a completely positive contractive map. Then there is a largest subalgebra $A_\phi\subseteq A$ such that $\phi|_{A_\phi}$ is a $*$-homomorphism. Furthermore, we have 
		\[\phi(ab)=\phi(a)\phi(b),\quad \phi(ba)=\phi(b)\phi(a),\quad \forall a\in A, b\in A_\phi.\]
\end{lem}
The subalgebra $A_\phi\subseteq A$ is called the \emph{multiplicative domain} of $A_\phi$.

\begin{prop}\label{Arveson-trick}
	Let $\mathcal G$ be an \'etale groupoid with unit space $X$. Denote by $S$ its inverse semigroup of open bisections. Suppose that $C^*\mathcal G=C^*_r\mathcal G$. Then there is an $S$-equivariant completely positive contractive map
		\[C_0(\beta_r \mathcal G)\to C_0(X)^{**}\]
	which extends the inclusion on $C_0(X)$. 
	\begin{proof}
		Let $\tilde \pi:C_0(X)^{**}\hookrightarrow \mathcal B(H)$ be the standard form of $C_0(X)^{**}$ and $(\tilde \pi,v)$ the associated covariant representation as in Lemma \ref{covarianthaagerup}. Denote by $\pi$ the restriction of $\tilde \pi$ to $C_0(X)$. Then $(\pi,v)$ integrates to a representation
			\[\pi\rtimes v:C^*_r\mathcal G=C^*\mathcal G=C_0(X)\rtimes S\to \mathcal B(H).\]
		Recall that reduced crossed products preserve inclusions. The proof of this fact in \cite[Lemma A.16]{echterhoff2002categorical} for groups can easily be adapted to groupoids. Thus, the inclusion $\iota:C_0(X)\to C_0(\beta_r \mathcal G)$ induces an inclusion
					\[\iota\rtimes_r \mathcal G:C^*_r \mathcal G\hookrightarrow C_0(\beta_r \mathcal G)\rtimes_r \mathcal G.\]
		By Arveson's extension theorem, there is a completely positive contractive map $\tilde \phi$ such that the following diagram commutes.
		\[\begin{tikzcd}
			C_0(\beta_r\mathcal G)\rtimes_r \mathcal G\arrow[dr,"\tilde \phi",dashed]&\\
			C^*_r\mathcal G\arrow[u,"\iota\rtimes_r \mathcal G",hook]\arrow[r,"\pi\rtimes v"] &\mathcal B(H)
		\end{tikzcd}\]		
%			\[\phi:C_0(\beta_r\mathcal G)\rtimes_r \mathcal G\to \mathcal B(H)\]
%		extending $\pi\rtimes v$.
		       Since $C_0(\beta_r \mathcal G)\subseteq C_0(\beta_r \mathcal G)\rtimes_r \mathcal G$ is commutative, it follows from Lemma \ref{multiplicativedomain} that $\tilde \phi(C_0(\beta_r \mathcal G))$ is contained in the commutant $ \pi(C_0(X))'\cong C_0(X)^{**}$. We claim that the restriction $\phi$ of $\tilde \phi$ to $C_0(\beta_r \mathcal G)$ is $S$-equivariant. Again by Lemma \ref{multiplicativedomain}, $\phi$ is $C_0(X)$-linear and thus preserves the domains of the partial actions by $S$.
		       
		        Now denote the canonical partial action of $S$ on $C_0(\beta_r \mathcal G)$ as well as its restriction to $C_0(X)$ by $\theta$. To see that $\phi$ is equivariant, fix elements $s\in S$ and $a\in C_0(\beta_r \mathcal G)_{s^*}$. We have to show that 
		        	\begin{equation}\label{phiisequivariant}
		        		\phi(\theta_s(a))=\theta_s^{**}(\phi(a))
		        	\end{equation}
		        holds. We identify $a$ with its image $a u_{s^*s}$ in the crossed product $ C_0(\beta_r\mathcal G)\rtimes S$ as in Remark \ref{inclusion-into-crossed-product}. Similarly, we identify $\theta_s(a)$ with $\theta_s(a)u_{ss^*}$ . Fix an approximate identity $(x_\lambda)_\lambda$ for $C_0(X)_{s}$. Since the inclusion $C_0(X)_{s}\to C_0(\beta_r\mathcal G)_{s}$ is non-degenerate, the image of $(x_\lambda)_\lambda$ in $C_0(\beta_r \mathcal G)_{s}$ is again an approximate identity. We denote it by $(x_\lambda)_\lambda$ as well. 
		        Similarly, $(\theta_{s^*}(x_\lambda))_\lambda$ is an approximate identity for $C_0(\beta_r\mathcal G)_{s^*}$. A calculation in the crossed product $ C_0(\beta_r\mathcal G)\rtimes S$ shows that we have 
		        	\begin{equation}\label{crossedproductcalc}
		        		x_\lambda u_s a \theta_{s^*}(x_\lambda)u_{s^*}=x_\lambda \theta_s(a)x_\lambda u_{ss^*}.
		        	\end{equation}
				Now reinterpret $\tilde \phi$ as a map on $C_0(\beta_r\mathcal G)\rtimes S$ by precomposing it with the quotient map 
					\[C_0(\beta_r\mathcal G)\rtimes S\cong C_0(\beta_r\mathcal G)\rtimes\mathcal G\to C_0(\beta_r\mathcal G)\rtimes_r\mathcal G.\]		        
		        Since the elements $x_\lambda u_s$ and $\theta_{s^*}(x_\lambda)u_{s^*}$ belong to the multiplicative domain of $\tilde \phi$ and since $\tilde \phi$ extends $\pi\rtimes v$, we obtain 
		        	\begin{align*}
		        		&\phi(\theta_s(a))=\tilde\phi(\theta_s(a)u_{ss^*})=\lim \tilde \phi(x_\lambda \theta_s(a)x_\lambda u_{ss^*})\\
		        		\overset{\eqref{crossedproductcalc}}=&\lim \tilde\phi(x_\lambda u_s a \theta_{s^*}(x_\lambda)u_{s^*})=\lim \pi(x_\lambda)v_s \phi(a)\pi(\theta_{s^*}(x_\lambda))v_{s^*}=v_s\phi(a)v_{s^*}\\
		        		=&\theta_s^{**}(\phi(a)).
		        	\end{align*}
		        This proves \eqref{phiisequivariant}.
	\end{proof}
\end{prop}

\begin{lem}\label{map-on-double-dual}
	Let $\theta$ be a partial action of an inverse semigroup $S$ on a $C^*$-algebra $A$. Then there is a $*$-homomorphism
		\[A^{**}\rtimes S\to (A\rtimes S)^{**}\]
	such that the composition 
		\[A\rtimes S\to A^{**}\rtimes S\to (A\rtimes S)^{**}\]
	is the canonical inclusion.
	\begin{proof}
		Represent $(A\rtimes S)^{**}\subseteq \mathcal B(H)$ faithfully, normally and non-degenerately on a Hilbert space $H$. The restriction of this inclusion to $A\rtimes S$ is an integrated form of a covariant representation $(\pi,v)$ by Proposition \ref{covariantrepresentationprop}. Denote by $\pi^{**}$ the unique normal extension of $\pi$ to $A^{**}$. We claim that $(\pi^{**},v)$ is again a covariant representation whose integrated form $\pi^{**}\rtimes v$ maps into $(A\rtimes S)^{**}$. Indeed, using normality of $\pi^{**}$ we get 
			\[\pi^{**}(A_s^{**})H=\pi(A_s)''H=\overline{\pi(A_s)H}=H_s\]
		for all $s\in S$. For an element $a\in A_s^{**}$ which is the ultraweak limit of a net $a_\lambda\in A_s$, we have
			\[\pi^{**}(\theta^{**}_s (a))=\lim \pi(\theta_s(a_\lambda))=\lim v_s \pi(a_\lambda)v_s^{*}= v_s\pi^{**}(a)v_s^*,\]
			again using normality of $\pi^{**}$ and $\theta^{**}_s$. Observe that we also have 
				\[\pi^{**}\rtimes v(a u_s)=\lim \pi(a_\lambda)v_s\in (A\rtimes S)^{**}.\]
			Thus the map $\pi^{**}\rtimes v$ has the desired properties. 
		
	\end{proof}
\end{lem}

We can now prove our main theorem:
\begin{thm}
	Let $\mathcal G$ be an \'etale groupoid which is strongly amenable at infinity. If $C^*_r\mathcal G=C^*\mathcal G$, then $C^*_r\mathcal G$ is nuclear. In particular $\mathcal G$ is amenable.

\end{thm}
\begin{proof}
	Denote by $S$ the inverse semigroup of open bisections of $\mathcal G$. By Proposition \ref{Arveson-trick}, there is an $S$-equivariant completely positive contractive map 
		\[\phi:C_0(\beta_r \mathcal G)\to C_0(X)^{**}\]
	extending the inclusion of $C_0(X)$. By Corollary \ref{completelypositive}, $\phi$ extends to a completely positive contractive map 
		\[\tilde \phi:C_0(\beta_r \mathcal G)\rtimes S\to C_0(X)^{**}\rtimes S.\]
	By Lemma \ref{map-on-double-dual}, there is a $*$-homomorphism 
		\[C_0(X)^{**}\rtimes S\to (C_0(X)\rtimes S)^{**}\]
		extending the inclusion on $C_0(X)\rtimes S$. 
		 Putting things together, we can express the inclusion $C_0(X)\rtimes S\hookrightarrow (C_0(X)\rtimes S)^{**}$ as the following composition of completely positive contractive maps:
		\begin{equation}\label{bigcomposition}
			C_0(X)\rtimes S\to C_0(\beta_r \mathcal G)\rtimes S\to C_0(X)^{**}\rtimes S\to (C_0(X)\rtimes S)^{**}.
		\end{equation}
	Since $\mathcal G$ was assumed to be strongly amenable at infinity, the $C^*$-algebra $C_0(\beta_r \mathcal G)\rtimes S\cong C_0(\beta_r \mathcal G)\rtimes \mathcal G\cong C_0(\beta_r \mathcal G)\rtimes_r \mathcal G$ is nuclear \cite[Proposition 7.2]{anantharaman2016exact}. Thus the map \eqref{bigcomposition} is nuclear. Recall that by \cite[Proposition 2.3.8]{brown2008textrm}, a $C^*$-algebra is nuclear if and only if its inclusion into its double dual is nuclear. Therefore $C_0(X)\rtimes S\cong C^*\mathcal G\cong C^*_r \mathcal G$ is nuclear. Now amenability of $\mathcal G$ follows from \cite[Corollary 6.2.14, Theorem 3.3.7]{anantharaman2001amenable}.
\end{proof}

\subsection*{Acknowledgements}
The author would like to thank Siegfried Echterhoff for several discussions and for introducing him to the topic of amenable actions. This work was funded by the Deutsche Forschungsgemeinschaft (DFG, German Research Foundation) – Project-ID 427320536 – SFB 1442, as well as under Germany’s Excellence Strategy EXC 2044 390685587, Mathematics Münster: Dynamics–Geometry–Structure.

%\textbf{Goal:} Reformulate the main theorem as follows:
%\begin{thm}
%	Let $S$ be an inverse semigroup acting on a commutative $C^*$-algebra $A$. Assume that $S$ acts on another $C^*$-algebra $B$ such that $A$ embeds equivariantly, centrally and non-degenerately $B$ such that $A\rtimes S\to B\rtimes S$ is injective. If $B\rtimes S$ is nuclear, then $A\rtimes S$ is nuclear as well. 
%\end{thm}
%Then derive the theorem about \'etale groupoids in a seperate chapter. 
\bibliography{JOT-submission}{}
	\bibliographystyle{abbrv}
\end{document}